\documentclass[a4paper,12pt]{article}
\input xy
\xyoption{all}
\usepackage{amssymb,amscd}
\usepackage{amsthm}
\usepackage{amsxtra}
\usepackage{mathrsfs}
\usepackage{array,float}
\usepackage{amsfonts}
\usepackage{xr}
\usepackage{color}

\begin{document}

\def\sect{\section}

\newtheorem{thm}{Theorem}[section]
\newtheorem{cor}[thm]{Corollary}
\newtheorem{lem}[thm]{Lemma}
\newtheorem{prop}[thm]{Proposition}
\newtheorem{propconstr}[thm]{Proposition-Construction}
\newtheorem{pro}[thm]{Proposition}

\theoremstyle{definition}
\newtheorem{para}[thm]{}
\newtheorem{ax}[thm]{Axiom}
\newtheorem{conj}[thm]{Conjecture}
\newtheorem{defn}[thm]{Definition}
\newtheorem{notation}[thm]{Notation}
\newtheorem{rem}[thm]{Remark}
\newtheorem{remark}[thm]{Remark}
\newtheorem{question}[thm]{Question}
\newtheorem{example}[thm]{Example}
\newtheorem{problem}[thm]{Problem}
\newtheorem{excercise}[thm]{Exercise}
\newtheorem{ex}[thm]{Exercise}
\newtheorem{fact}[thm]{Fact}

\newcommand{\Z}{\mathbb{Z}}
\newcommand{\Q}{\mathbb{Q}}
\newcommand{\N}{\mathbb{N}}

\def\mL{{\mathcal L}}
\def\C{{\mathcal C}}
\def\cF{{\mathcal F}}

\overfullrule=0pt

\def\si{\sigma}
\def\prf{\smallskip\noindent{\it         Proof}. }
\def\call{{\mathcal L}}
\def\nat{{\mathbb  N}}
\def\la{\langle}
\def\ra{\rangle}
\def\inv{^{-1}}
\def\ld{{\rm    ld}}
\def\trdeg{{tr.deg}}
\def\dim{{\rm   dim}}
\def\th{{\rm    Th}}
\def\rest{{\lower       .25     em      \hbox{$\vert$}}}
\def\ch{{\rm    char}}
\def\zee{{\mathbb  Z}}
\def\conc{^\frown}
\def\acl{acl_\si}
\def\cls{cl_\si}
\def\cals{{\cal S}}
\def\mult{{\rm  Mult}}
\def\calv{{\mathcal V}}
\def\aut{{\rm   Aut}}
\def\ffi{{\mathbb  F}}
\def\F{{\mathbb  F}}
\def\R{{\mathbb  R}}
\def\ffiti{\tilde{\mathbb          F}}
\def\degs{deg_\si}
\def\calx{{\mathcal X}}
\def\gal{{\mathcal G}al}
\def\cl{{\rm cl}}
\def\loc{{\rm locus}}
\def\calg{{\mathcal G}}
\def\H{{\mathcal H}}
\def\calq{{\mathcal Q}}
\def\calr{{\mathcal R}}
\def\caly{{\mathcal Y}}
\def\aff{{\mathbb A}}
\def\cali{{\cal I}}
\def\calu{{\cal U}}
\def\epsilon{\varepsilon} 
\def\U{{\mathcal U}}
\def\V{{\mathcal V}}
\def\rat{{\mathbb Q}}
\def\ga{{\mathbb G}_a}
\def\gm{{\mathbb G}_m}
\def\cee{{\mathbb C}}
\def\ree{{\mathbb R}}
\def\frob{{\rm Frob}}
\def\Frob{{\rm Frob}}
\def\fix{{\rm Fix}}
\def\Uu{{\mathcal U}}
\def\Kk{{\mathcal K}}
\def\proj{{\mathbb P}}
\def\sym{{\rm Sym}}
 
\def\dcl{{\rm dcl}}
\def\calm{{\mathcal M}}

\font\helpp=cmsy5
\def\semdp
{\hbox{$\times\kern-.23em\lower-.1em\hbox{\helpp\char'152}$}\,}

\def\dnfo{\,\raise.2em\hbox{$\,\mathrel|\kern-.9em\lower.35em\hbox{$\smile$}
$}}
\def\dnf#1{\lower1em\hbox{$\buildrel\dnfo\over{\scriptstyle #1}$}}
\def\dfo{\;\raise.2em\hbox{$\mathrel|\kern-.9em\lower.35em\hbox{$\smile$}
\kern-.7em\hbox{\char'57}$}\;}
\def\df#1{\lower1em\hbox{$\buildrel\dfo\over{\scriptstyle #1}$}}        
\def\stab{{\rm Stab}}
\def\qfcb{\hbox{qf-Cb}}
\def\perf{^{\rm perf}}
\def\sipm{\si^{\pm 1}}
\newcommand{\fgc}{\ffi_p[[G/C]]}
\newcommand{\fgone}{\ffi_p[[G/G_1]]}
\newcommand{\fgtwo}{\ffi_p[[G/G_2]]}
\newcommand{\fhc}{\ffi_p[[H/C]]}
\newcommand{\fhone}{\ffi_p[[H/H_1]]}
\newcommand{\fhtwo}{\ffi_p[[H/H_2]]}
\newcommand{\kfgc}{\ffi_p[[K\backslash G/C]]}
\newcommand{\kfgone}{\ffi_p[[K\backslash G/G_1]]}
\newcommand{\kfgtwo}{\ffi_p[[K\backslash G/G_2]]}
\newcommand{\kfhc}{\ffi_p[[K\cap H\backslash H/C]]}
\newcommand{\kfhone}{\ffi_p[[K\cap H\backslash H/H_1]]}
\newcommand{\kfhtwo}{\ffi_p[[K\cap H\backslash H/H_2]]}
\newcommand{\PZed}[1]{\textcolor{blue}{#1}}
\newcommand{\ZCed}[1]{\textcolor{green}{#1}}
\newcommand{\red}[1]{\textcolor{red}{#1}}
\newcommand{\G} {\mathcal G}
\newcommand{\cupdot}{\mathbin{\mathaccent\cdot\cup}}
\newcommand{\bigcupdot}{\hspace{9pt}\cdot \hspace{-9pt}\bigcup}
\newcommand{\PD} {\mathrm{PD}}

\def\vlabel{\label}

\title{A pro-$p$ version of Sela's accessibility  and  Poincar\'e duality pro-$p$ groups}

 \author{Ilaria Castellano and Pavel Zalesskii}

\date{}
\maketitle

\begin{abstract} We prove  a pro-$p$ version of Sela's theorem \cite{Sela} stating that a finitely generated
 group is $k$-acylindrically accessible. This result is then used to prove that $\PD^n$ pro-$p$ groups  admit a unique $k$-acylindrical JSJ-decomposition.

\end{abstract}

\section{Introduction}

Since 1970 the Bass-Serre theory of groups acting on trees stood out as one of the major advances in classical combinatorial group theory. The main notion of the Bass-Serre theory is the notion of  graph of groups. The fundamental group of a graph of groups acts naturally on a standard (universal) tree  that allows to describe subgroups of these constructions. 
This theory raised naturally the question of accessibility: namely whether we can continue to split  $G$ into an amalgamated free
product or an HNN-extension  forever, or do we reach the situation, after finitely many steps, where we can not split it anymore. In other words accessibility is the question whether splittings of $G$ as the fundamental group of a graph of groups have natural bound. Accessibility of splittings over finite groups (i.e., as a graph of groups with finite edge groups)  was studied by Dunwoody
(\cite{D85} and \cite{Dun93}) who proved that finitely presented groups are accessible but found an example of an inaccessible  finitely generated group. This initiated naturally a search for a kind of accessibility that holds for finitely generated groups. The breakthrough in this direction is due to Sela \cite{Sela} who proved $k$-acylindrical accessiblity for any finitely generated group: accessibility  provided 
the stabilizer of any segment of length $k$ of the group acting on its standard tree is trivial for some $k$.

\smallskip
The profinite version of Bass-Serre  theory was developed by Luis Ribes, Oleg Melnikov and the second author. However the pro-$p$ version of  Bass-Serre theory does not give subgroup structure theorems the way it does in the classical Bass-Serre theory: even in the pro-$p$ case, if $G$ acts on a pro-$p$ tree $T$ then a maximal subtree of the quotient graph $G\backslash T$ does not always exist and even if it exists it does not always lift to $T$. Nevertheless, the pro-$p$ version of the subgroups structure theorem works for pro-$p$ groups acting on a pro-$p$ trees that are accessible with respect to splitting over edge stabilizers; see \cite[Theorem 6.3]{CZ}. This shows additional  importance  of studying accessibility of pro-$p$ groups. In general finitely generated pro-$p$ groups are not accessible, as shown by G. Wilkes \cite{wilkes}, and it is an open question whether finitely presented are.
Our main result  in the direction is the pro-$p$ version of the celebrated Sela's result \cite{Sela} (cf. Theorem \ref{k-acylindrical accessibility}). 
\begin{thm}\label{k-acylindrical accessibility}  Let $G=\Pi_1(\G, \Gamma)$ be the fundamental group of a finite reduced $k$-acylindrical graph of pro-$p$ groups. Then $|E(\Gamma)|\leq d(G)(4k+1)-1 $, $|V(\Gamma)|\leq 4kd(G)$. 
\end{thm}
We use our accessibility theorem to establish the Kropholler type  \cite[Theorem A2]{K} JSJ-decomposition   for Poincar\'e duality pro-$p$ groups. JSJ decompositions first appeared in 3-dimensional topology with the theory of the characteristic submanifold by Jaco-Shalen and Johannson. These topological ideas were carried over to group theory first by Kropholler \cite{K} for some Poincar\'e duality groups. Later constructions of JSJ decompositions were given in various
settings by Sela for torsion-free hyperbolic groups \cite{Sel97b}, and in various settings by Rips-Sela \cite{RS97}, Bowditch \cite{Bow98}, Dunwoody-Sageev \cite{DS99}, Fujiwara-Papasoglu \cite{FP06}, Dunwoody-Swenson \cite{DS00}. . . ). This has had a vast influence and range of applications in geometric and combinatorial group theory.

The result below  can be considered as the first step towards this theory in the category of pro-$p$ groups. We establish a canonical JSJ-decomposition of Poincar\'e duality pro-$p$ groups of dimension  $n$ (i.e., $PD^n$ pro-$p$ groups)  which is a pro-$p$ version of the Kropholler \cite[Theorem A2]{K}. It also can be viewed as a pro-$p$ version of the torus decomposition theorem for 3-manifolds (cf. Theorem \ref{thm:JSJsection}). 
\begin{thm}\label{thm:JSJ}
	For every $PD^n$ pro-$p$ group $G$ ($n>2$) there exists a (possibly trivial) $k$-acylindrical pro-$p$ $G$-tree $\mathcal T$ satisfying the following properties:
	\begin{enumerate}
		\item[(i)] every edge stabilizer is a   maximal polycyclic subgroup  of $G$ of Hirsch length $n-1$;
		\item[(ii)] every polycyclic subgroup  of $G$ of Hirsch length $>1$ stabilizes a vertex;
		\item[(iii)] the underline graph of groups does not split further k-acylindrically over polycyclic subgroups  of $G$ of Hirsch length $n-1$.
	\end{enumerate}
	Moreover, every two pro-$p$
	$G$-trees satisfying the properties above are $G$-isomorphic.
\end{thm}
Examples of JSJ-decompositions of $PD^3$ pro-$p$ groups can be obtained by the pro-$p$ completion of   abstract JSJ-decomposition of some 3-manifolds (see \cite{wilkes17}). The pro-$p$ completion of $PD^n$ groups in general were studied in  \cite{KZ,W, HK,HKL, K}. 

\medskip
The structure of the paper is as follows. Section 2 recalls the notions of a pro-$p$ tree, a pro-$p$ fundamental group and a graph of pro-$p$ groups with a special focus on finite graphs of pro-$p$ groups. Throughout the paper finite graphs of pro-$p$ groups will be often required to be reduced and proper (see Definitions \ref{proper} and \ref{reduced}) but Remarks \ref{remark proper} and \ref{reduced-2} show that such an assumption is not restrictive.  Section 3 is devoted to the proof of the pro-$p$ version of Sela's accessibility which  states that every finitely generated pro-$p$ group is $k$-acylindrically accessible. Recall that a profinite graph of pro-$p$ groups $(\G,\Gamma)$ is {\em $k$-acylindrical} if the action of the pro-$p$ fundamental group  on its standard pro-$p$ tree  is $k$-acylindrical (cf. Subsection \ref{pro-p tree}). In this section we also prove the pro-$p$ version of Karras-Solitar result describing 2-generated subgroups of free products with malnormal amalgamation (see Theorem~\ref{thm:KS}). 
Finally Section 4 deals with  splittings of $PD^n$ pro-$p$ groups and culminates with a JSJ-decomposition for $PD^n$ pro-$p$ groups (see Theorem \ref{thm:JSJ})  which is a pro-$p$ version of the Kropholler \cite[Theorem A2]{K}. Note that the Kropholler theorem  \cite[Theorem A2]{K} gives also information  on vertex groups of a JSJ-splitting that is based on the Kropholler-Roller decomposition theorem \cite[Theorem B]{KR} that states that a $PD^n$ group $G$ having a $PD^{n-1}$ subgroup $H$ virtually splits  as a free product with amalgamation or HNN-extension over a subgroup commensurable with it if $cd(H\cap H^g)\neq  n-2$ for each $g\in G$. In fact, by   \cite[Theorem C]{KR},  $G$ virtually splits over $H$ if $H$ is polycyclic.

Unfortunately Kropholler-Roller theorems do not hold in the prop-$p$ case as shown by the following  example, which has been constructed in communication with Peter Kropholler during the visit of the second author to the University of Southampton.
\begin{example} Let $G$ be an open pro-$p$ subgroup of $SL_2(\Z_p)$ and $H$ is the intersection of the Borel subgroup of $SL_2(\Z_p)$ with $G$. Then $H$ is malnormal metacyclic subgroup of $G$ and therefore is $PD^2$ pro-$p$ group. The group $G$ is an analytitic pro-$p$ group of dimension 3 and so is a $PD^3$ pro-$p$ group. However, $G$ does not split as an amalgamated free  pro-$p$ product or HNN-extension at all.
\end{example}
In Section 5 we provide the details of the statement written in the example above. Here we just remark that the absence of the Kroholler-Roller splitting result is an obstacle of obtaining  information on vertex groups of a JSJ-splitting from Theorem \ref{thm:JSJ}.

\section{Notation, definitions and basic results}\vlabel{preliminaries}
\para{\bf Notation.} We shall denote by $d(G)$ the number of a minimal set of generators of a pro-$p$ group $G$ and by $\Phi(G)$ its Frattini subgroup.  If a pro-$p$ group $G$ continuously acts on a profinite space $X$ we denote by $G_x$ the stabilizer of $x$ in $G$.    
If $x\in X$ and $g\in G$, then $G_{gx}=gG_xg\inv$. We shall use the
notation  $h^g=g\inv hg$ for conjugation.  For  a subgroup $H$ of $G$,  $H^G$  will stand for the (topological) normal closure 
 of $H$ in $G$. If $G$ is an abstract group $\widehat G$ will mean the pro-$p$ completion of $G$. 
  
\para{\bf Conventions.} Throughout the paper, unless otherwise stated, groups are pro-$p$,
subgroups will be closed and 
morphisms will be continuous. Finite graphs of groups will be proper and
reduced (see Definitions \ref{proper} and \ref{reduced}). Actions of
a pro-$p$ group $G$  on a profinite graph $\Gamma$ will a priori be supposed to be faithful
(i.e., the action has no kernel), unless we consider actions on
subgraphs of $\Gamma$.  

 \bigskip
 Next  we collect basic definitions, following \cite{horizons}.

\subsection{Profinite graphs}
\begin{defn} A
{\em profinite graph} is a triple $(\Gamma, d_0, d_1)$, where
$\Gamma$ is a profinite (i.e. boolean) space and $d_0,d_1:\Gamma \to \Gamma $ are
continuous maps such that $d_id_j=d_j$ for $i, j \in \{0, 1 \}$.
The elements of $V(\Gamma):=d_0(G)\cup d_1(G)$ are called the
{\em vertices} of $\Gamma$ and the elements of
$E(\Gamma):=\Gamma\setminus V(\Gamma)$ are called the {\em edges} of 
$\Gamma$. If $e\in E(\Gamma)$, then $d_0(e)$ and $d_1(e)$ are
called the {\it initial} and {\it terminal vertices} of $e$. A vertex with only one incident edge is called {\it pending}. If there is no
confusion, one can just write $\Gamma$ instead of $(\Gamma, d_0,
d_1)$. \end{defn} 
\begin{defn}
A {\em morphism} $f:\Gamma \to \Delta$  of graphs is a map $f$ which
commutes with the $d_i$'s. Thus it will send vertices to vertices, but
might send an edge to a vertex.\footnote{It is called a {\em
    quasimorphism} in \cite{R 2017}.} 
\end{defn}
\begin{para} \label{collapse of edges}{\bf Collapsing edges.} We do not require for a morphism to send edges to edges. If $\Gamma$ is a graph and $e$ an edge which is not a loop we can {\em collapse} the edge $e$ by removing $\{e\}$ from the edge set of $\Gamma$, and
identify $d_0(e)$ and $d_1(e)$ with a new vertex $y$. I.e., $\Gamma'$ is
the graph given by $V(\Gamma')=V(\Gamma)\setminus \{d_0(e),d_1(e)\}\cup
\{y\}$ (where $y$ is the new vertex), and $E(\Gamma')=E(\Gamma)\setminus
\{e\}$. We define $\pi:\Gamma\to \Gamma'$ by setting $\pi(m)=m$ if
$m\notin \{e,d_0(e),d_1(e)\}$, $\pi(e)=\pi(d_0(e))=\pi(d_1(e))=y$. The
maps $d'_i:\Gamma'\to \Gamma'$ are defined so that $\pi$ is a morphism
of graphs. Another way of describing $\Gamma'$ is that
$\Gamma'=\Gamma/\Delta$, where $\Delta$ is the subgraph
$\{e,d_0(e),d_1(e)\}$ collapsed into the vertex $y$.
\end{para}
\begin{defn} Every profinite graph $\Gamma$ can be represented as an inverse limit $\Gamma=\varprojlim \Gamma_i$ of its finite quotient graphs (\cite[Proposition 1.5]{horizons}).

A profinite graph $\Gamma$ is said to be {\em connected} if all
its finite quotient graphs are connected. Every profinite graph is
an abstract graph, but a connected profinite graph is
not necessarily connected as an abstract graph. \end{defn}

A connected finite graph without circuits is called a {\it tree}. In the next subsection we shall explain how this notion extends to the pro-$p$ context. At the moment we shall prove several easy lemmas on finite graphs needed in the paper. The {\em valency} of a vertex is the number of edges connected to it. Hence, a vertex is pending if it has valency 1. A tree with two pending vertices will be called a {\em line}.

\subsection{Pro-$p$ trees}\label{pro-p tree}
\begin{para}{\bf The fundamental group of a profinite graph.}\label{fundamental}
	Let $\Gamma$ be  connected profinite graph. If $\Gamma=\varprojlim
	\Gamma_i$ is the inverse limit of the finite graphs $\Gamma_i$, then it induces the inverse
	system  $\{\pi_1(\Gamma_i)=\widehat \pi_1^{abs}(\Gamma_i)\}$ of the pro-$p$ completions of the abstract (usual) fundamental groups $\pi_1^{abs}(\Gamma_i)$. So 
	the pro-$p$ fundamental group $\pi_1(\Gamma)$ can be defined as $\pi_1(\Gamma)=\varprojlim_i
	\pi_1(\Gamma_i)$. If $\pi_1(\Gamma)=1$ then $\Gamma$ is called a {\em pro-$p$ tree}. 
\end{para}
 If $T$ is a pro-$p$ tree, then we say
  that a pro-$p$ group $G$ {\em acts on $T$} if it acts continuously on
  $T$ and the action commutes with $d_0$ and $d_1$. 

  If $t\in V(T)\cup E(T)$ we denote by $G_t$ the stabilizer of $t$ in $G$. For a pro-$p$ group $G$ acting on a pro-$p$ tree $T$ we let $\tilde{G}$
denote the subgroup generated by all vertex stabilizers. Moreover, for
any two vertices $v$ and $w$ of $T$ we let $[v,w]$ denote the geodesic connecting $v$ to $w$ in $T$, i.e., the (unique) smallest pro-$p$ subtree of $T$ that contains $v$ and $w$. The fundamental group $\pi_1(\Gamma)$ acts freely on a pro-$p$ tree $\widetilde \Gamma$ (universal cover) such that $\pi_1(\Gamma)\backslash \tilde\Gamma=\Gamma$ (see  \cite[Sedtion 3]{Z-89} or \cite[Chapter 3]{R 2017} for details).

An action of a pro-$p$ group on a pro-$p$  tree $T$ is called \emph{$k$-acylindrical} if the stabiliser of any geodesic in $T$ of length greater than $k$ is trivial. For instance, $0$-acylindrical refers to an action with trivial edge stabilisers, and $1$-acylindrical implies that edge stabilisers are malnormal in vertex-groups.
\begin{lem}\label{lem:malnormal}
Let $G$ be a pro-$p$ group acting $k$-acylindrically on a pro-$p$ tree $T$.  Then every polycyclic  subgroup $A$ of $G$ of Hirsch length $>1$  fixes a vertex. 
\end{lem}
\begin{proof}
	Let   $A\leq G$ be an polycyclic group of Hirsch length  $>1$. By contradiction assume that $A$ does not fix any vertex of $T$. By  \cite[Theorem 3.18]{horizons}) there exists a normal subgroup $N$ of $A$ stabilizing some vertex $v\in V(T)$. Since $A\neq A_v$,    the minimal subtree $T_A$ containing $Av$ is   fixed by $N$ (see \cite[Theorem 3.7]{horizons}). Since $T$ is $k$-acylindrical,  $T_A$ has diameter at most $k$, so   $A$ stabilizes a vertex. 
\end{proof}

\subsection{Finite graphs of pro-$p$ groups} 
In this subsection we recall the definition of a finite graph of pro-$p$ groups $(\G,\Gamma)$ and its fundamental pro-$p$ group $\Pi_1(\G, \Gamma)$. 
When we say that ${\cal G}$ is a finite graph of pro-$p$ groups we mean that it contains the data of the
underlying finite graph, the edge pro-$p$ groups, the vertex pro-$p$
groups and the attaching continuous maps. More precisely,
\begin{defn}
Let $\Gamma$ be a connected finite graph. A    graph of pro-$p$ groups $({\cal G},\Gamma)$ over
$\Gamma$ consists of  specifying a pro-$p$ group ${\cal G}(m)$ for each $m\in \Gamma$ (i.e. $\G= \bigcupdot_{m\in\Gamma} \G(m)$), and continuous monomorphisms
$\partial_i: {\cal G}(e)\longrightarrow {\cal G}(d_i(e))$ for each edge
$e\in E(\Gamma)$, $i=1,2$. 
\end{defn}
\begin{defn}
  \begin{enumerate}
\item A {\em morphism} of graphs of pro-$p$ groups:
$(\G,\Gamma) \rightarrow (\H,\Delta)$ is a pair 
$(\alpha,\bar\alpha)$  of maps, with  
 $\alpha:\G\longrightarrow\H$ a continuous map, and  $\bar\alpha:\Gamma\longrightarrow
\Delta$ a morphism of graphs, and such that $\alpha_{\G(m)}:\G(m)\longrightarrow
\H(\bar\alpha(m))$ is a homomorphism for each $m\in \Gamma$ and  which commutes with the appropriate $\partial _i$. Thus the diagram
$$\xymatrix{
\G\ar@{->}^\alpha[rr]\ar@{->}^{\partial_i}[d] & &\H\ar@{->}^{\partial_i}[d]\\
\G\ar@{->}^{\alpha}[rr] & &\H }$$ is commutative.
\item We say that $(\alpha, \bar\alpha)$ is a {\em monomorphism} if both $\alpha,\bar\alpha$ are injective. In this case its image will be called a {\em subgraph of groups} of $(\H, \Delta)$. In other words, a  subgraph of groups of  a graph of pro-$p$-groups
  $(\G,\Gamma)$ is a graph of groups $(\H,\Delta)$, where $\Delta$ is a
subgraph of $\Gamma$ (i.e., $E(\Delta)\subseteq E(\Gamma)$ and
$V(\Delta)\subseteq V(\Gamma)$, the maps $d_i$ on $\Delta$ are the
restrictions of the maps $d_i$ on $\Gamma$), and for each $m\in\Delta,$
$\H(m)\leq \G(m)$.
\end{enumerate}
\end{defn}
\begin{para}\label{fund graph groups} {\bf Definition of the fundamental pro-$p$ group.}  
In \cite[paragraph (3.3)]{Z-M 89b},  the fundamental group
 $G$ is  defined explicitly in terms of generators and relations
 associated to a chosen subtree $D$. Namely 
 \begin{equation} \label{presentation} G=\langle
 \G(v), t_e\mid v\in V(\Gamma), e\in E(\Gamma), t_e=1 \ {\rm for}\  e\in D, \partial_0(g)=t_e\partial_1(g)t_e^{-1},\  {\rm for}\ g\in \G(e)\rangle
\end{equation}
I.e., if one takes the abstract fundamental group $G_0=\pi_1(\G,\Gamma)$,
then $\Pi_1(\G,\Gamma)=\varprojlim_N G_0/N$, where $N$ ranges over
all normal subgroups of $G_0$ of index a power of $p$ and with $N\cap
\G(v)$ open in $\G(v)$ for all $v\in V(\Gamma)$. Note that this last
condition is automatic if $\G(v)$ is finitely generated (as a
pro-$p$-group) by \cite[\S 48]{RZ-10}. 
   It is also proved in \cite{Z-M 89b} 
that the definition given above is independent on the choice of
the maximal subtree $D$.\end{para}

The main examples of $\Pi_1(\G,\Gamma)$ are an amalgamated free pro-$p$
product $G_1\amalg_H G_2$ and an HNN-extension ${\rm HNN}(G,H,t)$ that correspond to the cases of $\Gamma$ having one edge and either two vertices or only one vertex, respectively. 
\begin{defn} \label{proper}We call the graph of groups $(\G,\Gamma)$ {\em proper}
  (injective in the terminology of \cite{R 2017}) if the natural map 
  $\G(v)\to \Pi_1(\G,\Gamma)$ is an embedding for all $v\in V(\Gamma)$.
  \end{defn}
\begin{rem}\label{remark proper} In the pro-$p$ case, a graph of groups $(\G,\Gamma)$ is not always {proper}. However, the
vertex and edge groups can always be replaced by their images  in
$\Pi_1(\G, \Gamma)$ so that $(\G,\Gamma)$ becomes proper and  $\Pi_1(\G,
\Gamma)$ does not change. Thus throughout the paper we shall only
consider  proper graphs of pro-$p$ groups. In particular, all our free amalgamated pro-$p$ products are proper. Thus we shall always identify vertex and edge groups of $(\G,\Gamma)$ with their images in $\Pi_1(\G,\Gamma)$. 
 \end{rem}
If $(\G,\Gamma)$ is a finite graph of
finitely generated pro-$p$ groups, then by a theorem of J-P.~Serre (stating that every finite index subgroup of a finitely generated pro-$p$ group is open, cf. \cite[\S 4.8]{RZ-10}) the  fundamental
pro-$p$ group $G=\Pi_1(\G,\Gamma)$ of
$(\G,\Gamma)$ is the pro-$p$ completion of the usual
fundamental group $\pi_1(\G,\Gamma)$ 
(cf. \cite[\S5.1]{Serre-1980}). Note that $(\G,\Gamma)$ is proper if and only if  $\pi_1(\G,\Gamma)$  is  residually $p$. In particular, edge and vertex groups will be
subgroups of $\Pi_1(\G,\Gamma)$. 
\begin{prop}\label{mod tilde U} Let $G=\Pi_1(\G,\Gamma)$ be the fundamental  pro-$p$  group of a finite proper graph of  pro-$p$ groups and $U$ a normal subgroup of $G$. Put $\tilde U=\langle \G(v)^g\cap U\mid g\in G,  v\in V(\Gamma)\rangle$. Then $\tilde U$ is normal in $G$ and $G/\tilde U=\Pi_1(\G_U, \Gamma)$, where $\G_U(m)=\G(m)U/U$ for each $m\in\Gamma$ with $\partial_0, \partial_1$  being natural inclusions in $G/U$.
\end{prop}
 \begin{proof}  The fundamental group $\Pi_1(\G_U, \Gamma)$ has a presentation \begin{equation} \label{presentation mod U} \langle
 \G_U(v), t_e\mid v\in V(\Gamma), e\in E(\Gamma), t_e=1 \ {\rm for}\  e\in D, \partial_0(g)=t_e\partial_1(g)t_e^{-1},\  {\rm for}\ g\in \G_U(e)\rangle
\end{equation}
 Therefore the kernel of the epimorphism $\Pi_1(\G,\Gamma)\longrightarrow \Pi_1(\G_U,\Gamma)$ induced by the natural morphism $(\G,\Gamma)\longrightarrow (\G_U,\Gamma)$ is generated as a normal subgroup by  $\G(v)\cap U, v\in V(\Gamma)$ as needed.
 \end{proof}
Let $(\G, \Gamma)$ be a profinite graph of pro-$p$ groups and $\Delta$ a subgraph of $\Gamma$. Then by $(\G, \Delta)$ we shall denote the graph of groups restricted to $\Delta$. We shall often use the following 
\begin{lem}\label{restricted graph} (\cite[Lemma 2.4]{SZ}) Let $(\G, \Gamma)$ be a proper finite graph   of pro-$p$ groups and $\Delta$ a connected subgraph of $\Gamma$. Then the natural homomorpism $\Pi_1(\G, \Delta)\longrightarrow \Pi_1(\G, \Gamma)$ is a monomorphism.
\end{lem}
\begin{prop}\label{free product} Let $G=\Pi_1(\G,\Gamma)$ be the fundamental group of a proper finite graph of pro-$p$ groups.  Suppose there exists an edge $e$ such that $G(e)=1$ and $G(d_i(e))\neq 1$ for $i=0,1$. Then $G$ splits as a free pro-$p$ product.
\end{prop}
\begin{proof} Suppose $\Gamma\setminus \{e\}$ is not connected.   
	 Then $\Pi_1(\G,\Gamma)=G_1\amalg G_2$, where $G_1$ and $G_2$ are the fundamental groups of the graphs of groups restricted to the connected components $C_1,C_2$ of $\Gamma\setminus \{e\}$ (cf. Lemma \ref{restricted graph}). So the result holds in this case.
	 
	 Otherwise, let $D$ be a maximal subtree of $\Gamma$ not containing an edge $e$. Then $G=HNN(G_1,G(e),t)$, where $G_1$ is the fundamental group of the graph of groups restricted to $\Gamma\setminus \{e\}$ (cf. Lemma \ref{restricted graph}). But since $G(e)=1$, we have $G=G_1\amalg \langle t\rangle$.
	\end{proof}
\begin{defn}\label{reduced} A finite graph of  pro-$p$ groups $(\G,\Gamma)$ is said to be {\em
reduced}, if for every  edge $e$ which is not a loop,
neither $\partial_{1}(e)\colon \G(e)\to \G(d_1(e))$ nor
$\partial_{0}(e):\G(e)\to \G(d_0(e))$ is an isomorphism. \end{defn}
\begin{rem} \label{reduced-2} Any finite graph of  pro-$p$ groups
can be transformed into a reduced finite graph of pro-$p$ groups by the
following procedure: If $\{e\}$ is an edge which is not a
loop and for which
one of $\partial_0$, $\partial_1$ is an isomorphism,  we can collapse
$\{e\}$ to a vertex $y$ (as explained in \ref{collapse of edges}). Let
$\Gamma^\prime$ be the finite graph given by
$V(\Gamma^\prime)=\{y\}\cupdot V(\Gamma)\setminus\{d_0(e),d_1(e)\}$
and $E(\Gamma^\prime)=E(\Gamma)\setminus\{e\}$, and let
$(\G^\prime, \Gamma^\prime)$ denote the finite graph of  groups
based on $\Gamma^\prime$ given by $\G^\prime(y)=\G(d_1(e))$ if
$\partial_{0}(e)$ is an isomorphism, and $\G^\prime(y)=\G(d_0(e))$ if
$\partial_{0}(e)$ is not an isomorphism. \\

This procedure can be
continued until $\partial_{0}(e), \partial_{1}(e)$ are not surjective for
all edges not defining loops. 
 Note that
the  reduction process does not change the
fundamental pro-$p$ group, i.e., one has a canonical isomorphism
$\Pi_1(\G,\Gamma)\simeq \Pi_1(\G_{red},\Gamma_{red})$.
So, if the pro-$p$ group $G$ is the fundamental group of a finite
graph of pro-$p$ groups, we may assume that the finite graph of
pro-$p$ groups is reduced. \\[0.05in]
\end{rem}
\begin{rem}\label{collapsed graph of groups}  The procedure of collapsing in the graph of pro-$p$ groups $(\G,\Gamma)$ can be generalized using Lemma \ref{restricted graph}. If $\Delta$ is a connected subgraph then we can collapse $\Delta$ to a vertex $v$ and put $G(v)=\Pi_1(\G,\Delta)$ leaving the rest of edge and vertex groups unchanged. The fundamental group $\Pi_1(\G_{\Delta},\Gamma/\Delta) =\Pi_1(\G,\Gamma)$. The graph of groups  $(\G_{\Delta},\Gamma/\Delta)$ will be called {\it collapsed}.
\end{rem}
\begin{lem} \label{bounding border} Let $G=\Pi_1(\G,\Gamma)$ be the
  fundamental  pro-$p$  group of a finite reduced tree of pro-$p$ groups
  $(\G,\Gamma)$ and let $d(G)$ be   the minimal number of generators of $G$. Then a minimal subset $V$ of $V(\Gamma)$ with $G=\langle \G(v)\mid v\in V\rangle$  contains all pending vertices of $\Gamma$ and has no more than $d(G)$ elements. 
\end{lem}
\begin{proof} For every pending vertex $v$ of $\Gamma$ and the
  (unique) edge $e\in \Gamma$ connected to it, $\overline\G(v)=\G(v)/\G(e)^{\G(v)}$
  is non-trivial, because the tree of groups $(\G, \Gamma)$ is reduced, and the groups are
  pro-$p$. Define the quotient tree of groups $(\overline \G, \Gamma)$
  by putting $\overline  \G(m)=1$ if $m\in \Gamma$ is not a pending vertex and $\overline\G(v)=\G(v)/\G(e)^{\G(v)}\neq 1$ if $v$ is a pending vertex.  Then from the presentation \eqref{presentation} for
  $\Pi_1(\overline \G, \Gamma)$ it follows then  that $$\Pi_1(\overline \G,
  \Gamma)=\coprod_{v\in V(\Gamma)} \overline\G(v)=\coprod_{v\in P_\Gamma} \overline\G(v),$$ where $P_\Gamma$ is the set of pending vertices of $\Gamma$. The natural morphism $(\G,\Gamma)\longrightarrow (\overline \G,\Gamma)$ induces then the epimorphism $G=\Pi_1(\G,\Gamma)\longrightarrow \overline G=\Pi_1(\overline \G,\Gamma)$. This shows that $P_\Gamma\subseteq V$.
  
  To show that $|V|\leq d(G)$ consider the Frattini quotient $\overline G=G/\Phi(G)$ and use overline for the images of subgroups of $G$ in $\overline G$. Since $d(G)=d(\overline G)$ and $\overline G$ is finite elementary abelian, one can choose finite $V_i\subseteq V$ with $|V_1|=1$ and $|V_{i+1}|=|V_i|+1$ such that  $\langle \overline\G(v)\mid v\in V_i\rangle$ is strictly increasing sequence of subgroups of $\overline G$. Then the number of terms in this sequence is $\leq d(G)$.
Hence  
   the number of  vertices of $V$ is at most 
  $d(\overline G)\leq d(G)$. 
\end{proof} 
\begin{prop}\label{reducing to trees}Let $G=\Pi_1(\G,\Gamma)$ be a finite graph of pro-$p$ groups and $D$ is a maximal subtree of $\Gamma$. Suppose $G$ is finitely generated.  If $\G(e)$ is finitely generated for every $e\in \Gamma\setminus D$, then $\Pi_1(\G, D)$ is finitely generated with $d(\Pi_1(\G, D))\leq d(G)+ \sum_{e\in \Gamma\setminus D} (d(\G(e)-1)$.
\end{prop}
\begin{proof} Since $\Gamma$ is finite, we can think of $G$ as $G=HNN(\Pi_1(\G, D), \G(e), t_e), e \in \Gamma\setminus D$. 

	Let $A=\Pi_1(\G, D)/\Phi(\Pi_1(\G, D))$ and $B$ be a subgroup generated by the images of $\G(e)$ in $A$ for $ e \in \Gamma\setminus D$. Since $\G(e)$ is finitely generated for each $ e \in \Gamma\setminus D$, the group $B$ is finite. Then there exists an epimorphism  $G=HNN(\Pi_1(\G, D), \G(e), t_e, e \in \Gamma\setminus D)\longrightarrow A/B\oplus \F_p[\Gamma\setminus D]$ that sends $\Pi_1(\G, D)$ to $A$ and $t_e$ to $e$ in the vector space $\F_p[\Gamma\setminus D]$.  Since $A/B\oplus \F_p[\Gamma\setminus D]$  is finite,  $A/B$ is finite and so $A$ is finite implying that  $\Pi_1(\G, D)$ is finitely generated. Since $d(\Pi_1(\G, D))=d(A)=d(A/B)+d(B)$, $d(G)\geq d(A/B) +|\Gamma\setminus D|= d(A)-d(B)+|\Gamma\setminus D|$, one deduces $d(\Pi_1(\G, D))=d(A)\leq d(G)-|\Gamma\setminus D|+d(B)\leq d(G)+ \sum_{e\in \Gamma\setminus D} (d(\G(e)-1)$.
\end{proof}
\begin{para}{\bf Standard (universal) pro-$p$ tree.}\label{standard}
Associated with the finite graph of pro-$p$ groups $({\cal G}, \Gamma)$ there is
a corresponding  {\em  standard pro-$p$ tree}  (or universal covering graph)
  $T=T(G)=\bigcupdot_{m\in \Gamma}
G/\G(m)$ (cf. \cite[Proposition 3.8]{Z-M 89b}).  The vertices of
$T$ are those cosets of the form
$g\G(v)$, with $v\in V(\Gamma)$
and $g\in G$; its edges are the cosets of the form $g\G(e)$, with $e\in
E(\Gamma)$; and the incidence maps of $T$ are given by the formulas:

$$d_0 (g\G(e))= g\G(d_0(e)); \quad  d_1(g\G(e))=gt_e\G(d_1(e)) \ \ 
(e\in E(\Gamma), t_e=1\hbox{ if }e\in D).  $$

 There is a natural  continuous action of
 $G$ on $T$, and clearly $ G\backslash T= \Gamma$. 
Remark also that since $\Gamma$ is finite, $E(T)$ is compact.
\end{para}

\section{Acylindrical accessibility}
In this section we shall prove  a pro-$p$ version of Sela's accessibility.
 Note that Sela used $\R$-trees for the proof; later Weidmann \cite[Theorem 4]{W} found another proof using Nielsen method and established a bound. Both methods are not available in the pro-$p$ case.

We shall start  with two auxiliary results on free amalgamated product and its generalization for abstract groups.

\begin{lem} \label{abstract amalgam} Let $G=G_1*_H G_2$ be a splitting of a  group as an amalgamated free  product and $H_1\leq G_1, H_2\leq G_2$. Then  $\langle H_1,H_2\rangle=L_1 *_K L_2$, where $L_1=\langle H_1, H_2\cap H\rangle$ and $L_2=\langle H_2, H_1\cap H\rangle$ and $K=\langle H_1\cap H,H_2\cap H\rangle$. In particular, if  $H_1\cap H\leq U\geq H\cap H_2$ for some  normal subgroup $U$ of $G$ then $L_1\leq H_1(U\cap G_1), L_2\leq H_2(U\cap G_2), K\leq H\cap U$.
\end{lem}
\begin{proof}  First note that  it follows from the Bass-Serre theory \cite{Serre-1980} that $\langle H_1,H_2\rangle$ is a free amalgamated product whose factors are contained in $G_1$ and $G_2$, respectively. To see this it suffices to consider the Bass-Serre tree $T$ associated to $G$ and denote by $e$ the edge whose vertices have stabilizers $G_1$ and $G_2$, respectively. Now one notices that  the $\langle H_1,H_2\rangle$-orbit of $e$ in $T$ is connected, and it provides a tree acted on by $\langle H_1,H_2\rangle$ with a single edge orbit. 
	
	Therefore we need to prove that the factors  of
	the splitting are $L_1$ and $L_2$ and the amalgamated subgroup is $K$. To this end we claim that an element  $x\in \langle H_1,H_2\rangle$ has   a reduced form  $h=x_1x_2\ldots x_n$  with $x_i\in L_1\cup L_2$. Suppose not and $x=a_1a_2\cdots a_m$ be an expression as a product of the minimal length of alternating elements from $H_1$ or $H_2$ (i.e. if $a_i\in H_1$ then $a_{i+1}\in H_2$) such that a reduced word of it is not of the desired form. Then a reduced word for $a_2\cdots a_n$ has a reduced form  $a_2\cdots a_m=l_1\cdots l_k$ with $l_i\in L_1\cup L_2$.

	Recall that $a_1\in H_i\leq L_i$ for $i=1$ or 2.  Since the word
	$a_1l_1 \cdots l_k$ is not  reduced and $l_1\cdots l_k$ is,  the reduction happens in $a_1l_1$ that can occur in the free amalgamated product $G = G_1 *_H G_2$
	 only if $a_1, l_1 \in H$. In particular, either $a_1 \in H_1 \cap H$ or $a_1 \in H_2 \cap H$ and so $a_1l_1 \cdots l_k$ is a reduced word of needed form if  $a_1$ and $l_1$ belong to different $L_i$s;  if  $a_1$ and $l_1$ belong to the same $L_i$s, then the consolidated word $(a_1l_1) \cdots l_k$ has
	 entries from $L_1 \cup L_2$ and is reduced.  This gives a  contradiction.
	 
	 It remains to prove that $K=\langle H_1\cap H,H_2\cap H\rangle$.  For $k\in K$ write  minimal expressions $k=x_1\cdots x_n$ and $k=y_1\cdots y_m$ as alternating  products of elements of $H_1, H_2\cap H$ and $H_2, H_1\cap H$ respectively.   Thus $x_1\cdots x_n=y_1\cdots y_m$. If $k\not\in \langle H_1\cap H,H_2\cap H\rangle$ then there are $x_i,y_j\not\in H$ for some $i,j$ and we can choose $i$ maximal and $j$ minimal with this property. But then the product $y_m^{-1}\cdots y_1^{-1}x_1\cdots x_m$ can not be reduced to 1, since $y_j^{-1}$ and $x_i$ can not be canceled. 	 
\end{proof}
\begin{prop} \label{general} Let $G=\pi_1(\G,\Gamma)$ be the fundamental group of a  tree $H_v\leq G(v)$of groups  for $v\in V(\Gamma)$. Then $H=\langle H_v\mid v\in V(\Gamma)\rangle=\pi_1(\mL, \Gamma)$ such that $L(v)=\langle H_v, G(e)\cap H_w\rangle$ and $L(e)=\langle H_v\cap G(e),H_w\cap G(e)\rangle$, where $e$ ranges over the edges incident to $v$ and $w$ is the other vertex of $e$. In particular, if $U$ is a normal subgroup of $G$ and,   for each edge  $e$ and its vertex $v$, one has $H_v\cap G(e)\leq U$ then   $L(v)\leq H_vU$ and  $L(e)\leq U\cap G(e)$.  
\end{prop}
\begin{proof} We use induction on $|\Gamma|$. If $\Gamma$ has one edge only, the result follows from Lemma \ref{abstract amalgam}.  Let $e$ be an edge of $\Gamma$ having $w$ as a pending vertex. Then $G=G_1*_{G_e} G_w$.  Let $v$ be the other vertex of $e$ and put $H'_v=\langle H_v, H_w\cap G(e)\rangle$. 
	Let $H_1=\langle H_u, H'_v\mid u\in V(\Gamma)\setminus \{v\}\rangle$.  By the induction hypothesis $H_1=\pi_1(\mL_1, \Delta)$, with $\Delta=\Gamma\setminus \{e,w\}$ and  vertex and edge groups satisfying the statement of the proposition.  Applying Lemma \ref{abstract amalgam} we get $\langle H_1, H_w\rangle=L_1*_K L(w)$, where $L_1=\langle H_1, G(w)\cap G(e)\rangle, L(w)=\langle H_w, H_1\cap G(e)\rangle$ and $K=\langle H_1\cap G(e), H_w\cap G(e)\rangle$. 
It follows that $H=\langle H_v\mid v\in V(\Gamma)\rangle=\langle H_1, H_w\rangle=\langle H_1, \langle H_u, H_w\cap G(e)\rangle\rangle=\pi_1(\mL,\Delta)\amalg_K L(w)=\pi_1(\mL,\Gamma)$ with the desired   properties.
\end{proof}
\begin{lem}  Let $G=G_1\amalg_H G_2$ be a splitting of a  pro-$p$ group as an amalgamated free  product of finite groups and $H_1\leq G_1, H_2\leq G_2$ be subgroups such that $H_1\cap H\leq U\geq H\cap H_2$ for some open normal subgroup $U$ of $G$. Then  $\langle H_1,H_2\rangle=L_1 \amalg_K L_2$ with $L_1\leq H_1U, L_2\leq H_2U, K\leq HU$.
\end{lem}
\begin{proof} By \cite[Proposition 4.4]{ZM90}, $\langle H_1,H_2\rangle=L_1 \amalg_K L_2$.
	with $L_1\leq G_1$, $L_2\leq G_2$, $K\leq H$. By Lemma \ref{abstract amalgam} combined with paragraph  \ref{fund graph groups} $L_1\leq H_1U, L_2\leq H_2U, K\leq HU$.
\end{proof}
\begin{cor} \label{amalgam}  Let $G=G_1\amalg_H G_2$ be a splitting of a  pro-$p$ group $G$  as an amalgamated free  pro-$p$ product  of pro-$p$  groups $G_1,G_2$ and $H_1\leq G_1, H_2\leq G_2$ be subgroups such that $H_1\cap H=1= H_2\cap H$. Then    $\langle H_1,H_2\rangle=H_1 \amalg H_2$. \end{cor}
\begin{proof}  Since $U$ in the preceding lemma is arbitrary, the result follows.
\end{proof}

\begin{prop}\label{intersection} Let $G=\Pi_1(\G,\Gamma)$ be the fundamental pro-$p$ group of a finite tree  of pro-$p$ groups and $H_v\leq G(v)$ for $v\in V(\Gamma)$. Let $U$ be an open normal subgroup of $G$ and suppose  that  for each edge  $e$ one has $H_v\cap G(e)\leq U\geq H_w\cap G(e)$. Then $H=\langle H_v\mid v\in V\rangle=\Pi_1(\H, \Gamma)$ such that $H(v)\leq H_vU$ for all $v\in V(\gamma)$ and  $H(e)\leq U\cap \G(e)$ for all $e\in E(\Gamma)$.  
	
\end{prop}

\begin{proof} By \cite[Proposition 4.4]{ZM90}, $\langle H_v\mid v\in V(\Gamma)\rangle=\Pi_1(\H,\Gamma)$ 
	with $H(m)\leq G(m)$. By Proposition \ref{general} comnbined with patagraph \ref{fund graph groups}, $H(v)\leq H_vU$ and $H(e)\leq U\cap G(e)$.
	\end{proof}
\begin{cor}\label{free product}  Suppose  $H_v\cap G(e)=1$ for all  $v\in V(\Gamma)$ and each $e\in E(\Gamma)$. Then $H=\coprod_{v\in V(\Gamma)} H_v$.
\end{cor}
\begin{proof}  Since $U$ in Proposition \ref{intersection} is an arbitrary open normal subgroup, $H(v)=H_v$ and $\H(e)=1$ for each $e\in D$. Hence $H=\coprod_{v\in V(\Gamma)} H_v$ by \cite[Example 6.2.3]{R 2017}. 
\end{proof}
\begin{defn} We say that a profinite graph of pro-$p$ groups $(\G,\Gamma)$ is {\em $k$-acylindrical} if the action of the fundamental groups $\Pi_1(\G,\Gamma)$ on its standard pro-$p$ tree  is $k$-acylindrical. \end{defn}
\begin{prop}\label{distance} Let $G=\Pi_1(\G,\Gamma)$ be the fundamental pro-$p$ group of an acylindrical graph of pro-$p$ groups. Let $v,w$ be  vertices at distance $\geq 2k+1$. Then $\langle G(v), G(w)\rangle=G(v)\amalg G(w)$. 
\end{prop}
\begin{proof}
	Let $[v,w]$ be the shortest geodesic between $v$ and $w$. Let $G(v,w)$ be the fundamental group of the graph of pro-$p$ groups restricted to $[v,w]$. By Lemma \ref{restricted graph} $G(v,w)$ is a subgroup of $G$ generated by vertex stabilizers of $[v,w]$.  Let $e$ be an edge of $[v,w]$  at distance $>k$ from   $w$ and $v$. Then $G(v)\cap G(e)=1=G(w)\cap G(e)$. Note that $G(v,w)$ splits over $G(e)$ as a free amalgamated pro-$p$ product $G(v,w)=G_1\amalg_{G(e)} G_2$, where $G_1$, $G_2$ are  pro-$p$ groups generated by vertex groups of the connected components  of $[v,w]\setminus e$ (see Lemma \ref{restricted graph}), so that  $G(v)\leq G_1$ and $G(w)\leq G_2$. By Corollary   \ref{amalgam}, $\langle G(v),G(w)\rangle=G(v)\amalg G(w)$ as required.
\end{proof}

\begin{prop}\label{at most 2}  Suppose $\Gamma=[v,w]$ be a line of pro-$p$ groups such that $G=\Pi_1(\G,\Gamma)=G(v)\amalg G(w)$.  Let $(\G,\Gamma_{red})$ be a reduced graph of pro-$p$ groups obtained from $(\G,\Gamma)$ by the procedure described in Remark \ref{reduced-2}. If $\Gamma_{red}$ is not a vertex, then one of the following holds:
	\begin{enumerate}
		\item[(i)] $\Gamma_{red}$ has only two edges $e_1,e_2$ with pending vertices $v,w$ and one middle vertex $u$ such that $G(u)=G(e_1)\amalg G(e_2)$;	
		\item[(ii)] $\Gamma_{red}$ has only one edge, a trivial edge group and $G(v), G(w)$ as vertex groups.	
	\end{enumerate}
\end{prop}
\begin{proof}   Let $U$ be an open normal subgroup of $G$ and $G_U(v)=G(v)U/U, G_U(w)=G(w)U/U$. Let $U(v,w)=(\langle U\cap G(v), U\cap G(w)\rangle)^G$ and $G_U=G/U(v,w)=G_U(v)\amalg G_U(w)$ (cf. Proposition \ref{mod tilde U}).  Then $G=\varprojlim_U G_U$ where $G_U=\Pi_1(\G_U,\Gamma_{red})$ and $\G_U=\bigcup G(m)U(v,w)/U(v,w)$ for every $m\in \Gamma$. Starting with some $U$, the graph of groups $(\G_U,\Gamma_{red})$ is reduced, and w.l.o.g. we may assume that it is reduced for every $U$.
	
	Suppose $\Gamma_{red}$ has one edge. Then $G_U=G_U(v)\amalg_{G_U(e)} G_U(w)$. It follows that $G_U(e)=1$ for each $U$ and therefore so is $G(e)$. 
	
	Suppose now $\Gamma_{red}$ has more than one edge; we shall use induction on the sum  $|G_U(v)|+|G_U(w)|$ of the orders of the free factors of $G_U=G_U(v)\amalg G_U(w)$ to show that $\Gamma_{red}$ satisfies (i) or (ii). 
	Let $e_1,e_2$ be edges of $\Gamma$ incident to $v$ and $w$ respectively, and $v_1$, $w_1$ the other vertices of $e_1$ and $e_2$.  By the pro-$p$ version of the Kurosh Subgroup Theorem, one has 
	\begin{eqnarray}\nonumber
		\Pi_1(\G_U,[v_1,w_1])&=&(G_U(v)\cap \Pi_1(\G_U,[v_1,w_1]))\amalg (\Pi_1(\G_U,[v_1,w_1]))\cap G_U(w))\amalg L\\ \nonumber
		&=& G_U(e_1)\amalg G_U(e_2)\amalg L
	\end{eqnarray}
	and so $G_U=G_U(v)\amalg L\amalg G_U(w)$. Hence $L=1$ and $\Pi_1(\G_U,[v_1,w_1])= G_U(e_1)\amalg G_U(e_2)$. If $v_1=w_1$ then we are in case (i). Suppose $v_1\neq w_1$. By induction hypothesis, $(\G_U,[v_1,w_1])$ satisfy (i) or (ii) and so in either case    $G_U(v_1)=G_U(e_1)$ and $G_U(w_1)=G_U(e_2)$. Hence edges $e_1$ and $e_2$ are fictitious, a contradiction. Hence $v_1=w_1$. 
	Thus putting $u=v_1=w_1$ we have $G_U(u)=G_U(e_1)\amalg G_U(e_2)$ and so $G(u)=G(e_1)\amalg G(e_2)$.
\end{proof}
\begin{prop}\label{free pro-p product} Let $G=\Pi_1(\G,\Gamma)$ be the fundamental pro-$p$ group of a $k$-acylindrical  finite tree of pro-$p$ groups. Suppose there exists a subset $V\subset V(\Gamma)$ such that 
	\begin{enumerate}
		\item[(i)] $[v,w]$ has at least $2k+1$ edges whenever $v\neq w\in V$; 
		\item[(ii)] $G=\langle G(v)\mid v\in V\rangle$. 
	\end{enumerate}
	Then $G=\coprod_{v\in V} G(v)$.
\end{prop}
\begin{proof}  For every $v\in $V, we collapse the ball of radius $k$ centered at $v$ to the vertex $v$ itself and we consider the collapsed graph of groups $\Gamma'$ obtained in this way from Remark \ref{collapsed graph of groups}. Setting $H_v=G_v$ for $v\in V$ and $H_v=1$ for $v\notin V$, we achieve premises of Corollary \ref{free product}, since the action is $k$-acylindrical, deducing from it the result.
\end{proof}
\begin{cor}\label{acylindrical line}  Let $G=\Pi_1(\G,\Gamma)$ be the fundamental pro-$p$ group of a reduced $k$-acylindrical  finite line of pro-$p$ groups ($k>0)$. Let $V$ be the minimal subset of $V(\Gamma)$ such that $G=\langle G(v)\mid v\in V\rangle$. If $G$ is  finitely generated then $|E(\Gamma)|\leq 2k|V|$. 
\end{cor}
\begin{proof}   We just need to show that the distance between two neighboring vertices of $V$ is at most $2k$. Suppose on the contrary $v,w$ are neighboring vertices of $V$ such that $[v,w]$ has at least $2k+1$ edges.    Collapsing connected components $C_v$ and $C_w$ of $\Gamma\setminus ]v,w[$ and considering  the collapsed graph of pro-$p$ groups (see Remark \ref{collapsed graph of groups}) instead of $(\G,\Gamma)$ we may assume that $\Gamma=[v,w]$. By Proposition \ref{distance}, $G=G(v)\amalg G(w)$.  But	then Proposition \ref{at most 2} forces  $[v,w]$ to have at most two edges, a contradiction.  
\end{proof}
\begin{cor}\label{free splitting} Let $G=\Pi_1(\G,\Gamma)$ be the fundamental group of a proper finite $k$-acylindrical  tree of pro-$p$ groups. Let $V$ be the minimal subset of $V(\Gamma)$ such that $G=\langle G(v)\mid v\in V\rangle$.  Suppose there exists a vertex  $v\in V$ such that the distance $l(v,w)$ is at least $2k+1$ for every $w\in  V$. Then $G$ splits as a free pro-$p$ product.
\end{cor}
\begin{proof}  Divide $V$ as the disjoint union $\{v\}\cup \bigcup_{i=1}^lV_i$ where the sets $V_i$ are defined as follows: $V_i=V\cap C_i$ where $C_i$ is a connected component of $\Gamma\setminus B(v, 2k)$, where $B(v,2k)$ is the  ball  of radius $2k$ with the center in $v$.  Denote by $\Delta_i$ the span of $V_i$ and let $G_i=\Pi_1(G,\Delta_i)$ be the fundamental group of a graph of groups restricted to $\Delta_i$. 
	Using Remark \ref{collapsed graph of groups} we can collapse all  $\Delta_i$. The obtained graph of groups satisfies premises of Proposition \ref{free pro-p product} and by hypothesis possesses more then one vertex. Hence,  by Proposition \ref{free pro-p product},  it is a non-trivial free pro-$p$ product.	
\end{proof}
\begin{thm}\label{k-acylindrical accessibility}  Let $G=\Pi_1(\G, \Gamma)$ be the fundamental group of a finite reduced $k$-acylindrical graph of pro-$p$ groups.  Then $|E(\Gamma)|\leq d(G)(4k+1)-1 $ and $|V(\Gamma)|\leq 4kd(G)$. 
\end{thm}
\begin{proof} Let $D$ be a maximal subtree of $\Gamma$. By \cite[Lemma~3.6]{CZ},  there are at most $d(G)$ edges in 
	$\Gamma\setminus D$.   Let $V$ be a minimal subset of 
	$V(\Gamma)$ such that $G=\langle G(v),t_e\mid v\in V, e\in \Gamma\setminus D\rangle$. Looking at $G/\Phi(G)$ one easily deduces that  $|V|\leq d(G)$. Let $V'$ be the set of vertices connected to vertices of $V$ by an edge $e\in \Gamma\setminus D$. Then $V\cup V'\leq 2d(G)$ as follows from presentation of $\Pi_1(\G,D)=\langle G(v)\mid v\in V\cup V'\rangle$; indeed if not, then we can factor out all these $G(v)$s and get a non-trivial free product $\pi_1(\Gamma)\amalg L$ for some $L$ that contradicts $G=\langle G(v),t_e\mid v\in V, e\in \Gamma\setminus D\rangle$.  By Corollary \ref{free splitting},   the distance between vertices of $V\cup V'$ is at most $2k$. Hence the number of vertices in $D$ is at most $4kd(G)$ and therefore the number of edges of $\Gamma$ is at most $4kd(G)-1+d(G)=d(G)(4k+1)-1 $.
\end{proof}
\begin{cor}\label{cor:coherent} Let $G$  be a free amalgamated pro-$p$ product $G=G_1\amalg_H G_2$  of coherent pro-$p$ groups over an analytic pro-$p$ group $H$. If $H$ is malnormal in $G_1$ then $G$ is coherent.
\end{cor}
\begin{proof}  Let $K$ be a finitely generated subgroup of $G$. Then $K$ acts at most  $2$-acylindrically on the standard pro-$p$ tree $T(G)$.  By Theorem \ref{k-acylindrical accessibility}, $K$ is $2$-acylindrically accessible. By \cite[Theorem 3.6]{CZ},  $K=\Pi_1(\H,\Gamma)$ is the fundamental group of a finite graph of finite pro-$p$ groups with edge groups being conjugate to subgroups of $H$. Hence, for each edge $e\in\Gamma$, one has  $d(\H(e))\leq rank(H)$, where $rank(H)$ means the Pr\"ufer rank. Therefore $K$ is finitely presented (cf. (\ref{presentation})).
\end{proof}
\begin{thm}\label{bound}  Let $G=\Pi_1(\G, \Gamma)$ be the fundamental group of a finite reduced $k$-acylindrical graph of pro-$p$ groups with $d(G(e))\leq n$ for each $e\in E(\Gamma)$. Suppose  $G$ is finitely generated.  Then $|E(\Gamma)|\leq  (2kn+1)d(G)$.   
\end{thm}
\begin{proof} Let $D$ be a maximal subtree of $\Gamma$. By \cite[Lemma~3.6]{CZ},  there are at most $d(G)$ edges in $\Gamma\setminus D$. By Proposition \ref{reducing to trees}, $d(\Pi_1(\G,D))\leq d(G)+(n-1)d(G)=nd(G)$. By Lemma \ref{bounding border},  $D$ has  at most $nd(G)$ pending vertices in $D$.  Let $V$ be a minimal set  of vertices such that $\Pi_1(\G,D)=\langle G(v)\mid v\in V\rangle$. Then $|V|\leq d(\Pi_1(\G,D))$ and so, by Corollary \ref{acylindrical line},  $|E(D)|\leq 2knd(G)$. So $|E(\Gamma)|\leq (2kn+1)d(G)$.
\end{proof}
\begin{cor}
	Suppose all edge groups are  2-generated and $k=1$ . Then $|E(\Gamma)|\leq 5d(G)$.
\end{cor}
We finish the section  with a pro-$p$ version of the Karras-Solitar \cite[Theorem~6]{KS} but we start with the lemma below where generation symbols $\langle\rangle$  mean abstract generation  unlike in the rest of the paper where $\langle\rangle$ means topological generation.
\begin{lem}\label{malnormal} 
	Let $G=G_1\amalg_HG_2$ be a non-fictitious free  pro-$p$ product with malnormal amalgamation. Suppose $G$ is 2-generated. Then $H$ is trivial and $G_1,G_2$ are cyclic.
\end{lem}
\begin{proof}  
	Let $x\in G_1, y\in G_2$ such that  $G$ is generated by $ x$ and $y$. Consider the abstract subgroup   $\langle x,y\rangle$ of the abstract free amalgamated product $G_1*_HG_2$. By  \cite[Theorem 6]{KS},  $\langle x,y\rangle$  is a free product $\langle x\rangle *\langle y\rangle$. By Lemma \ref{abstract amalgam}, $\langle x\rangle\cap H=1=H\cap \langle y\rangle$. Hence  $\overline{\langle x\rangle}\cap H=1=\overline{\langle y\rangle}\cap H$ and, by Corollary \ref{amalgam},  $G=\overline{\langle x\rangle}\amalg \overline{\langle y\rangle}$. Thus the result follows from Proposition \ref{free pro-p product} (ii).
\end{proof}
\begin{thm}\label{thm:KS}
	Let $G=G_1\amalg_HG_2$ be a free  pro-$p$ product with malnormal amalgamation and $K$ is 2-generated subgroup of $G$. If $K$ is not conjugate to a subgroup of $G_1$ or $G_2$, then $K$ is a free pro-$p$ product of two cyclic groups.
\end{thm}
\begin{proof} Consider the action of $K$ on the standard pro-$p$ tree $T(G)$. Then the action is acylindrical. We assume that $K$ does not stabilize a vertex (if it does it conjugates into $G_1$ or $G_2$). 
	Suppose first that $K$ is generated by vertex stabilizers. By \cite[Theorem 4.2]{CZ} and its proof (see Case 1 there),  there exists   a non-trivial splitting $K=K_1\amalg_{K_e} K_2$ as a free pro-$p$  product with amalgamation over an edge stabilizer. Then the result follows from Corollary \ref{malnormal}. Suppose now $K$ is not generated by vertex stabilizers. By \cite[Theorem 4.2]{CZ} and its proof (see Case 2 there),  there exists  a non-trivial splitting $K=HNN(L,K_e,t)$ as a pro-$p$ HNN-extension over an edge stabilizer. 
	Note that $K=\langle x,t\rangle$ such that $x\in L$  and $t$ is the stable letter. By \cite[Theorem~4.2]{CZ} and its proof (see Case 1 there), $R=\langle x, x^t\rangle=R_1\amalg_{R_e} R_2$ and every vertex-group of $R$ belongs either to $R_1$ or $R_2$ up to conjugation. It follows that $x$ and $x^t$ belong to different factors. Then, by Corollary \ref{malnormal}, $R=\langle x\rangle \amalg \langle x^t\rangle $ is a free pro-$p$ product. It follows that $K=\langle x\rangle \amalg \langle t\rangle$ as needed.
\end{proof}

\section{Decomposing $PD^n$ pro-$p$ groups}
\subsection{Pro-$p$ $PD^n$-pairs}  In \cite{wilkes2} Wilkes defined the profinite version of group pairs but we shall need only a simple version of it. A pro-$p$ group pair $(G,\mathcal S)$ consists of a pro-$p$ group $G$ and a finite family $\mathcal S$ of closed subgroups $S_x$ of $G$ indexed over a set (we allow repetitions in this family). 
Given a closed subgroup $H$ of $G$, let $\mathcal S^H$ denote the family of subgroups
\begin{equation}\label{eq:SH}
	\{H\cap \sigma(y)S_x\sigma(y)^{-1}\mid x\in X, y\in H\backslash G/S_x\},
\end{equation}
indexed over 
$$H\backslash G/\mathcal S:=\bigsqcup_{x\in X} H\backslash G/S_x,$$
where $\sigma\colon G/H\to G$ is a section of the quotient map $G\to G/H$\footnote{A different section only affects the family $\mathcal S^H$
	by changing its members by conjugacy in $H$.}.

In \cite{wilkes2} the author develops the theory of the cohomology of a profinite group relative to a collection of closed subgroups and defines profinite Poincar\'e duality pairs (or $PD^n$-pairs for short) and the reader is referred to \cite[Section~5]{wilkes2} for rigorous definitions and basic results. A pro-$p$ group pair $(G,\mathcal S)$ is a pro-$p$ $PD^n$-pair, for some $n\in\mathbb N$, if the double
of $G$ over the groups in $\mathcal S$ is a pro-$p$ $PD^n$-group. Here the double of $G$ over $\mathcal S$ refers to the fundamental group of a graph of groups with two vertices and $|\mathcal S|$ edges where a copy of $G$ is over each vertex and groups of $\mathcal S$ are over the edges, with natural boudary maps. 
\begin{example}\label{ex:subpairs}
	Let $G$ be a $PD^n$ pro-$p$ group   isomorphic to the fundamental group of a reduced proper finite graph of pro-$p$ groups $(\mathcal G,\Gamma)$ whose edge-groups are $PD^{n-1}$ sungroups of $G$. 
	For each vertex $v\in V(\Gamma)$ denote by $\mathcal E_v$ the collection of all the subgroups of $G(v)$ which are  images $\partial_i(G(e))$ of those edge-groups such that $d_i(e)=v$. Then $(G(v),\mathcal E_v)$ is a pro-$p$ $PD^n$-pair by \cite[Theorem~5.18(2)]{wilkes2} for $\mathcal S=\emptyset$.
\end{example}
We say that a  pro-$p$ $PD^n$-pair $(G,\mathcal S)$ splits as an amalgamated free pro-$p$ product $G=G_1\amalg_H G_2$ (resp. as HNN-extension $HNN(G_1,H,t)$) if each $S_i$ is conjugate to either $G_1$ or $G_2$ (resp. $G_1$).

The next proposition was communicated to us by G. Wilkes.
\begin{prop}\label{prop:non split} (G. Wilkes) Let $(G,\mathcal S)$ be a pro-$p$ $PD^n$-pair with $\mathcal S=\{S_1,\ldots,S_n\}$. Then, for every $i=1,\ldots,n$, $(G,\mathcal S)$ does not split over $S_i$.
\end{prop}
The proof relies on the following 
\begin{lem}\label{lem:2copies} Let $G$ be a pro-$p$ group such that $(G,\mathcal S)$ is a $PD^n$-pair. Suppose $S_1= S_2$. Then  $m=2$ and $S_1=S_2=G$.
\end{lem}
\begin{proof} By \cite[Theorem 5.17(1)]{wilkes2}, the pro-$p$ HNN-extension $\tilde G=HNN(G,S_1=S_2,t)$ with $s^t=s$ for $s \in S_1$ is a $PD^n$-pair relative to the collection $\{S_3,\ldots,S_n\}$. Since $\tilde G$ contains the pro-$p$ $PD^n$-group $S_1\times \langle t\rangle$ (cf. \cite[Proposition 5.9]{wilkes2}), one has $\mathrm{cd}_p(\tilde G)=n$. By \cite[Corollary 5.8]{wilkes2}, $\{S_3,\ldots,S_n\}$ is empty and $m=2$.
	
	If $G\neq S_1$, take an open subgroup $U$ containing $S_1$. If $\mathcal S^U$ is the collection defined in \eqref{eq:SH}, then $(U,\mathcal S^U)$ is a $PD^n$-pair
	(see the proof of \cite[Proposition 5.11]{wilkes2}). But $|\mathcal S^U|=2|U\backslash G/S_1|>2$ and $\{S_1,S_2\}\subset\mathcal S^U$, contradicting the first part.
\end{proof}
\begin{proof}[Proof of Proposition \ref{prop:non split}] Suppose by contradiction that $(G,\mathcal S)$ does split along some $S_i$. Assume w.l.o.g  $i=1$. Up to changing $\mathcal S$ by conjugacy, $G$ is either isomorphic to $HNN(G_1,S_1,t)$, with $S_k\leq G_1$ for every $k=1,\ldots,n$, or isomorphic to $G_1\amalg_H G_2$  with $S_k\leq G_1$ or $S_k\leq G_2$ for every $k=1,\ldots,n$. In the latter case, $\mathcal S$ can be decomposed as $\mathcal S_1\sqcup\mathcal S_2$ where each $\mathcal S_j$  ($j=1,2$) contains only elements from $\mathcal S$ which are also subgroups of $G_j$. Assume w.l.o.g. that $S_1\in \mathcal S_1$. Then by  \cite[Theorem 5.16(2)]{wilkes2} for $G\cong G_1\amalg_H G_2$ and by  \cite[Theorem 5.17(2)]{wilkes2} if $G\cong HNN(G_1,H,t)$  the pair  $(G_1,\mathcal S\sqcup\{H\})$ is a $PD^n$-pair which contradicts Lemma \ref{lem:2copies}.
\end{proof}

\subsection{Splitting over polycyclic subgroups}
Here we collect some results that will be used later in the proof of the main theorem. 

We say that a pro-$p$ group $G$ admits a {\em $k$-acylidrical  splitting} if $G$ is isomorphic to the fundamental pro-$p$ group $\Pi_1(\mathcal G,\Gamma)$ of a  $k$-acylindrical  proper reduced finite graph of pro-$p$ groups. 
\begin{prop}\label{prop:acyl}
	Let $G$ be a pro-$p$ $PD^n$-group which is the fundamental group $\Pi_1(\G,\Gamma)$ of a finite  reduced graph of pro-$p$ groups with   $PD^{n-1}$ edge subgroup of $G$. Then the stabilizers of two adjacent edges of $T$ are not commensurable.
\end{prop}
\begin{proof} 
	 We just need to show that  two adjacent edge groups $G(e_1)$, $G(e_2)$  do not intersect by a subgroup of finite index. Suppose they do. Then there exists an open subgroup $U$ of $G$ such that $U\cap G(e_1)=G(e_1)\cap G(e_2)=U\cap G(e_2)$. So replacing $G$ by $U$ we may assume that $G(e_1)=G(e_2)$. Let $v$ be their common vertex. By  Example \ref{ex:subpairs}, $(G(v), \mathcal E_v)$ is a pro-$p$ $PD^n$ pair with $G(e_1), G(e_2)\in \mathcal E_v$ contradicting  Lemma \ref{lem:2copies}.
\end{proof}
The next theorem establishes a pro-$p$ version JSJ-decomposition for $PD^n$ pro-$p$ groups analogous of one from \cite[Theorem A2]{K}.
\begin{thm}\label{thm:JSJsection}
	For every $PD^n$ pro-$p$ group $G$, ($n>2$) there exists a (possibly trivial)  $k$-acylindrical pro-$p$ $G$-tree $\mathcal T$ satisfying the following properties:
	\begin{enumerate}
		\item[(i)] every edge stabilizer is a   maximal polycyclic subgroup  of $G$ of Hirsch length $n-1$;
		\item[(ii)] polycyclic subgroup  of $G$ of Hirsch length $>1$ stabilizes a vertex;
		\item[(iii)] the underline graph of groups does not split further k-acylindrically over a polycyclic subgroup  of $G$ of Hirsch length $n-1$.
	\end{enumerate}
	Moreover, every two pro-$p$
	$G$-trees satisfying the properties above are $G$-isomorphic.
\end{thm}
\begin{proof}
	 By Theorem   \ref{bound}, a k-acylindrical  decomposition as  fundamental group of a reduced finite graph of pro-$p$ groups $(\mathcal G,\Lambda)$ with  polycyclic subgroup  of $G$ of Hirsch length $n-1$ has a bound, so we can choose one with a maximal number of edge groups. In particular, the edge-groups satisfy property (i).
	
	 By Lemma~\ref{lem:malnormal}, the standard pro-$p$ tree $\mathcal T$ also satisfies propery (ii). We shall show now property (iii). 
	
	First notice that the vertex-stabilizers of $\mathcal T$ cannot decompose k-acylindrically over  polycyclic subgroups  of Hirsch length $>1$ at all. Indeed suppose on the contrary that some vertex-group $G(v)$ of $(\mathcal G,\Lambda)$ splits $k$-acylindrically either as $G_0\amalg_A G_1$ or $HNN(G_0,A,t)$, where $A$ is polycyclic of Hirsch length $>1$. Then, by Lemma~\ref{lem:malnormal}, the edge-groups of all adjacent edges to $G(v)$ are conjugate into either $G_0$ or $G_1$. Denote by $E_i$ the set of edges in $\mathrm{star}_{\Lambda}(v)$ whose edge group is conjugate into $G_i$ with $i=0,1$. Thus we can replace the vertex $v$ by an edge $e$ with two vertices $v_1$ and $v_2$, connecting the edges $e_i\in E_i$ to $v_i$, together with boundary maps  $\partial_i\colon G(e_i)\to G(v_i)$ given by correspondent conjugation for every $e_i\in E_i$. Note that the construction of this map is continuous because $\mathrm{star}_{\Lambda}(v)$ is finite. This contradicts the maximality of the decomposition. 
	
	Given any two trees $\mathcal T$ and $\bar{\mathcal T}$ satisfying the  properties (i)-(iii), we claim that there exists a $G$-equivariant morphism $\phi\colon \mathcal T\to\bar{\mathcal T}$. Let us prove the claim. In order to construct $\phi$ we need to map $G$-equivariantly each edge $e$ of $\mathcal T$ to an edge $\bar e$ of $\bar{\mathcal T}$. Let $v=d_0(e)$ and $w=d_1(e)$. 
	Therefore, there exist vertices $\bar v$ and $\bar w$ such that $G_v\subseteq G_{\bar v}$ and $G_w\subseteq G_{\bar w}$. Hence it suffices to prove that $\bar v$ and $\bar w$ are at distance 1 in the tree $\bar{\mathcal T}$ and set $\phi(e)=\bar e$, where $\bar e$ denotes the edge connecting $\bar v$ to $\bar w$. By Proposition \ref{prop:acyl} , edge-groups of distinct edges in $\Lambda$ are not commensurable. Therefore,  one sees that $G(v_1)\cap G(v_2)$ is polycyclic subgroup  of $G$ of Hirsch length $n-1$ that  implies adjacency.
\end{proof}
The uniqueness of $\mathcal T$ in Theorem \ref{thm:JSJ} induces an action on it by the automorphism group $Aut(G)$. This gives a splitting sturcture on $Aut(G)$ if $\mathcal T$ is non-trivial. We state this as a 
\begin{cor}
The automorphism group $Aut(G)$ acts on $\mathcal T$. Moreover, if $\mathcal T$ is not a vertex then $Aut(G)$ splits as non-trivial amalgamated free pro-$p$ product or pro-$p$ $HNN$-extension.
\end{cor}

\section{Example}

\begin{thm} Let $G$ be an abstract $PD^n$ group and $H$ its $PD^{n-1}$ subgroup.  
	
	\begin{enumerate}
		\item[(i)] (\cite[Theorem B ]{KR}) Suppose that $\mathrm{cd}(H\cap H^g)\neq  n-2$ for each $g\in G$. Then $G$ splits as an amalgamated free product or HNN-extension over a group commensurable with $H$. 
		\item[(ii)] (\cite[Theorem C ]{KR}) Suppose $H$ is polycylic. Then some  finite index subgroup of $G$ splits as an amalgamated free product or HNN-extension over a group commensurable with $H$.
	\end{enumerate}

\end{thm}

The next example shows that both Kropholler-Roller theorems do not hold in the pro-$p$ case, i.e.  neither of the statements of the theorems above.

\begin{example}  Let $G$ be an open subgroup of $SL_2(\Z_p)$ (say the first congruence pro-$p$ subgroup)  and $H=B\cap G$ is the intersection of the Borel subgroup $B$ of $SL_2(\Z_p)$ with $G$. Then $H$ is a maximal metacyclic subgroup of $G$ and therefore is $PD^2$ pro-$p$ group. Moreover, $H$ is malnormal. Indeed, the group of upper unipotent matrices  is a normal subgroup of $H$ which is  isomorphic to $\Z_p$ on which a subgroup of diagonal matrices 
$	\begin{pmatrix} t&0\\
	0&t^{-1}\end{pmatrix}$
		acts as multiplication by $t^2$; recalling that the group of units of $\Z_p$ is isomorphic to $\Z_p\times C_{p-1}$ (if $p>2$) and to $\Z_2\times C_2$ (if $p=2$), we see that $H$ is metacyclic, say $H=U\rtimes T$, where $U$ consists of  unipotent elements and $T$ consists of diagonal elements. 
	
	To see that $H$ is malnormal in $G$ consider $A=H\cap H^g$ for some $g\in G\setminus B$. First observe that a straightforward calculation shows that for the unipotent upper triangle group $U$ one has $U\cap U^g=1$   for $g\not\in B$. Now  if $B\cap B^g$ intersects  $U$ non-trivially, then this intersection is cyclic, since otherwise it is open in $B$ contradicting the preceding sentence.  Moreover, it is normal in both $B$ and $B^g$ and so in $\langle B,g\rangle$ by  \cite[Lemma 15.2.1 (a)]{R 2017}. But $U^g$ is a unique maximal normal subgroup in $B^g$ so $B\cap B^g\cap U\leq U^g\cap U=1$.  It follows that  $B\cap B^g$ is  generated by a semisimple element $s$ and,  as it is not scalar  (the scalars have order 2), its eigenvalues  are disjoint   $t,t^{-1}$.  
	This matrix $s$  has two 1-dimensional eigen submodules of $\Z_p\oplus \Z_p$: $V_t$  associated with eigenvalue $t$ and $V_{t^{-1}}$   associated with eigenvalue $t^{-1}$. Hence $V_{t}\cap V_{t^{-1}}=0$.  Note that if $v\in V_t$ then $v/p\in \Z_p\oplus \Z_p$ implies $v/p\in V_t$. This means that $V_t/p\neq V_{t^{-1}}/p$ in $\F_p\oplus \F_p$. But $g$ is trivial modulo $p$ and  so $gV_{t}/p= V_{t^{-1}}/p$, a contradiction.

	 The group $G$ is an analytitic pro-$p$ group of dimension 3 and so is a $PD^3$ pro-$p$ group. It has no non-abelian pro-$p$ subgroups and it is not soluble. So by \cite[Theorem 4.7 and 4.8]{horizons} it does not split as an amalgamated free pro-$p$ product or HNN-extension.
	
\end{example}

\end{document}